\documentclass[11pt]{amsart}
\usepackage{amsxtra}
\usepackage{amssymb}
\usepackage{enumerate}
\font\wasy=wasy10 at 10pt
\theoremstyle{plain}
\newcommand{\cleqn}{\setcounter{equation}{0}}
\newcommand{\clth}{\setcounter{theorem}{0}}
\newcommand {\sectionnew}[1]{\section{#1}\cleqn\clth}
\newtheorem{theorem}{Theorem}[section]
\newtheorem{lemma}[theorem]{Lemma}
\newtheorem{definition-theorem}[theorem]{Definition-Theorem}
\newtheorem{proposition}[theorem]{Proposition}
\newtheorem{corollary}[theorem]{Corollary}

\theoremstyle{definition}
\newtheorem{remark}[theorem]{Remark}
\newtheorem{definition}[theorem]{Definition}
\newtheorem{example}[theorem]{Example}
\newtheorem*{remark*}{Remark}

\newtheorem{observation}[theorem]{Observation}
\newcommand \KK {{\mathbb K}}

\newcommand \Zset {{\mathbb Z}}
\newcommand \de {\delta}
\newcommand \al {\alpha}

\newcommand \la {\lambda}

\newcommand \sig {\sigma}
\DeclareMathOperator \Fract { {\mathrm{Fract}} }

\newcommand\kx{\KK^*}

\newcommand\HH{{\mathcal{H}}}
\newcommand\xh{X(\HH)}
\DeclareMathOperator \Spec {Spec}

\newcommand \Znn {\Zset_{\ge 0}}

\newcommand \Rhat {\widehat{R}}

\newcommand\Hspec{\HH{\text{-}}\Spec}
\newcommand{\gc}{ [ \hspace{-0.65mm} [}
\newcommand{\dc}{]  \hspace{-0.65mm} ]}

\DeclareMathOperator\hgt{ht}
\DeclareMathOperator\pht{P-ht}

\DeclareMathOperator\GK{GKdim}

\newcommand\hfrak{\mathfrak{h}}
\DeclareMathOperator\Lie{Lie}
\DeclareMathOperator \Pspec {P{.}Spec}
\newcommand \HPspec {\HH\text{-}{\Pspec}\,}

\newcommand\E{\mathcal{E}}
\DeclareMathOperator\Kdim{K{.}dim}

\begin{document}

\title{Poisson catenarity in Poisson nilpotent algebras}

\author[K. R. Goodearl]{K. R. Goodearl}
\address{
Department of Mathematics \\
University of California\\
Santa Barbara, CA 93106 \\
U.S.A.
}
\email{goodearl@math.ucsb.edu}

\author[S. Launois]{S. Launois}
\address{
School of Mathematics, Statistics and Actuarial Science \\
University of Kent \\
Canterbury, Kent, CT2 7FS \\
 UK
}
\email{S.Launois@kent.ac.uk}

\begin{abstract}
We prove that for the iterated Poisson polynomial rings known as Poisson nilpotent algebras (or Poisson-CGL extensions), the Poisson prime spectrum is catenary, i.e., all saturated chains of inclusions of Poisson prime ideals between any two given Poisson prime ideals have the same length.
\end{abstract}



\thanks{The research of the first named author was supported
by US National Science Foundation grant DMS-1601184. That of the second named author was supported by EPSRC grant EP/N034449/1.}

\maketitle

%%%%%%%%%%%%%%%%%%%%%%%%%%

%%%%%%%%%%%%%%%%%%%%%%%%%
\sectionnew{Introduction}

The aim of this paper is to study the Poisson prime spectra of Poisson algebras. More precisely,  we focus on the catenary property -- that all saturated chains of inclusions of Poisson prime ideals between any two fixed Poisson prime ideals have the same length -- for the large class of Poisson algebras called Poisson nilpotent algebras. Many such algebras arise as semiclassical limits of quantum algebras, e.g., coordinate rings of matrix varieties or Schubert cells with natural Poisson structures. See \cite[Section 2]{KGSL} and \cite{EL}, for instance, for a number of examples.

Many properties of quantum algebras are known, or are conjectured to be, reflected in parallel properties of the Poisson algebras that arise as their semiclassical limits. Catenarity has been proved for the prime spectra of many quantized coordinate rings and related algebras, and the authors recently proved it for quantum nilpotent algebras \cite{KGSL2}. The parallel result, for the Poisson prime spectra of Poisson nilpotent algebras, is established here.

In order to provide further detail, a few definitions are in order. 

\subsection{Poisson nilpotent algebras}

The algebras just named are iterated Poisson polynomial algebras $A[x;\sig,\de]_p$ (see Definition \ref{Ppoly}), meaning that $A$ is a Poisson algebra and $A[x;\sig,\de]_p$ denotes a polynomial ring $A[x]$ equipped with a Poisson bracket that extends the one on $A$ and satisfies
$$
\{x,a\} = \sig(a) x + \de(a) \qquad \forall\, a\in A,
$$
where $\sig$ and $\de$ are suitable derivations on $A$. Given an iterated Poisson polynomial algebra
\begin{equation}  \label{itPOre}
R := \KK [x_1]_p [x_2;\sig_2,\de_2]_p \cdots [x_N; \sig_N,\de_N]_p \,,
\end{equation}
we set
$$
R_k := \KK[x_1,\dots,x_k] = \KK [x_1]_p [x_2;\sig_2,\de_2]_p \cdots [x_k; \sig_k,\de_k]_p 
$$
for $k \in \gc0,N\dc$. In particular, $R_0 = \KK$ and $R_1 = \KK[x_1]_p = \KK[x_1; 0,0]_p$.

\begin{definition} An iterated Poisson polynomial algebra $R$ as in \eqref{itPOre} is called a \emph{Poisson nilpotent algebra} or a \emph{Poisson-CGL extension} \cite[Definition 5.1]{GY.pcluster} if it is equipped with a rational action of an algebraic torus $\HH$ by Poisson algebra automorphisms such that
\begin{enumerate}
\item[(i)] The elements $x_1,\dots,x_N$ are $\HH$-eigenvectors.
\item[(ii)] For every $k \in \gc2,N\dc$, the map $\de_k$ on $R_{k-1}$ is locally nilpotent.
\item[(iii)] For every $k \in \gc1,N\dc$, there exists $h_k \in \Lie\HH$ such that $(h_k{\cdot})|_{R_{k-1}} = \sig_k$ and the $h_k$-eigenvalue of $x_k$ is nonzero.
\end{enumerate}
\end{definition}

The main theorem of the paper is

\begin{theorem}  \label{mainthm}
If $R$ is a Poisson nilpotent algebra, then its Poisson prime spectrum $\Pspec R$ is catenary.
\end{theorem}

Here $\Pspec R$ consists of the ideals of $R$ which are both Poisson ideals and prime ideals (see Definition \ref{P.misc} and Lemma \ref{Pprimeinfo}).

\subsection{Notation and conventions} 
Throughout, all algebras will be unital commutative algebras over a fixed base field $\KK$ of characteristic zero.

%%%%%%%%%%%%%%%%%%%%%%%%%
\sectionnew{Background and basics}

Recall that a \emph{Poisson algebra} (over $\KK$) is a commutative $\KK$-algebra equipped with a Poisson bracket $\{-,-\}$, that is, a Lie bracket which is also a derivation in each variable (for the associative product). 

\begin{definition}  \label{Ppoly}
Let $A$ be a Poisson algebra. A \emph{Poisson derivation} on $A$ is a $\KK$-linear map $\sig$ on $A$ which is a derivation with respect to both the
multiplication and the Poisson bracket, that is, 
$$
\sig(ab)= \sig(a)b+
a\sig(b) \qquad\text{and}\qquad \sig(\{a,b\})= \{\sig(a),b\}+ \{a,\sig(b)\} \qquad \forall\, a,b\in A.
$$
Suppose that $\delta$ is a \emph{Poisson $\sigma$-derivation} on $A$ (in the terminology of \cite[\S1.1.2]{Dumas}), meaning that $\delta$ is a $\KK$-linear derivation and
\begin{equation}  \label{delta.cond}
\delta(\{a,b\})= \{\delta(a),b\}+ \{a,\delta(b)\}+ \sig(a)\delta(b)-
\delta(a)\sig(b) \qquad \forall\, a,b\in A.
\end{equation}
By \cite[Theorem 1.1]{Oh} (after replacing $\sig$ with $-\sig$), the Poisson structure on $A$ extends
uniquely to a Poisson algebra structure on the polynomial ring $A[x]$
such that
$$
\{x,a\}= \sig(a)x+ \delta(a) \qquad \forall\, a\in A.
$$
We write $A[x]=
A[x;\sig,\delta]_p$ to denote this situation, and we refer to $A[x;\sig,\delta]_p$ as a
\emph{Poisson polynomial algebra} or a \emph{Poisson-Ore extension}. The Poisson structure on $A[x]$ extends uniquely to the Laurent polynomial ring $A[x^{\pm1}]$. We write $A[x^{\pm1};\sig,\delta]_p$ for this Poisson algebra and refer to it as a \emph{Poisson Laurent polynomial algebra}.

In either of the above cases, we omit $\delta$ from the notation if it is zero,
that is, we write $A[x;\sig]_p$ and $A[x^{\pm1};\sig]_p$ for
$A[x;\sig,0]_p$ and $A[x^{\pm1};\sig,0]_p$.

The converse part of \cite[Theorem 1.1]{Oh} will also be needed: If a polynomial ring $A[x]$ supports a Poisson bracket such that $A$ is a Poisson subalgebra and $\{x,A\} \subseteq A+Ax$, then $A[x]= A[x;\sig,\delta]_p$ for suitable $\sig$ and $\delta$.
\end{definition}

\begin{definition}  \label{P.misc}
Let $R$ be a Poisson algebra. The Poisson bracket on $R$ induces unique Poisson brackets on any quotient of $R$ modulo  a \emph{Poisson ideal}, meaning an ideal $I$ such that $\{R,I\} \subseteq I$, and on any localization of $R$ with respect to a multiplicative set (e.g.,
\cite[Proposition 1.7]{Loose}). 

A \emph{Poisson automorphism} of $R$ is any algebra automorphism which preserves the Poisson bracket. We use the term \emph{Poisson action} to refer to an action of a group on $R$ by Poisson automorphisms. 

The \emph{Poisson center} of $R$ is the set $Z_p := \{ z \in R \mid \{z,-\} \equiv 0 \}$, which is a subalgebra of $R$. A \emph{Poisson-prime} ideal is any proper Poisson ideal $P$ of $R$ such that $(IJ \subseteq P \implies I \subseteq P \; \text{or} \; J \subseteq P)$ for all Poisson ideals $I$ and $J$ of $R$.
\end{definition}

\begin{lemma}  \label{Pprimeinfo}  \cite[Lemma 1.1]{GpDixMo} 
Let $R$ be a Poisson algebra.

{\rm(a)} Every prime ideal minimal over a Poisson ideal is a Poisson ideal.

{\rm(b)} If $R$ is noetherian, every Poisson-prime ideal of $R$ is prime.
\end{lemma}

\begin{remark*}
If $R$ is a noetherian Poisson algebra, Lemma \ref{Pprimeinfo}(b) implies that the Poisson-prime ideals
 in $R$ are precisely the ideals which are both
Poisson ideals and prime ideals; in that case, the hyphen in the term 
``Poisson-prime'' becomes unnecessary.
\end{remark*}

The \emph{Poisson prime spectrum} of a Poisson algebra $R$ is the set $\Pspec R$ of all Poisson-prime ideals of $R$, equipped with the natural Zariski topology (where the closed sets are the sets $\{ P \in \Pspec R \mid P \supseteq I \}$ for Poisson ideals $I$ of $R$).

Suppose $R$ is a Poisson algebra and we have a Poisson action of a group $\HH$ on $R$. One defines \emph{$\HH$-prime} and \emph{$\HH$-Poisson-prime} ideals of $R$ in the same manner as Poisson-prime ideals. Namely, an $\HH$-prime (resp., $\HH$-Poisson-prime) ideal is any proper $\HH$-stable (resp., $\HH$-stable Poisson) ideal $P$ of $R$ such that $(IJ \subseteq P \implies I \subseteq P \; \text{or} \; J \subseteq P)$ for all $\HH$-stable (resp., $\HH$-stable Poisson) ideals $I$ and $J$ of $R$. We denote the sets of $\HH$-prime and $\HH$-Poisson-prime ideals of $R$ by $\Hspec R$ and $\HPspec R$, respectively.

These concepts are related by an analog of Lemma \ref{Pprimeinfo} when $\HH$ is an algebraic torus and the action of $\HH$ on $R$ is \emph{rational}, meaning that $R$ is spanned by $\HH$-eigenvectors whose $\HH$-eigenvalues $\HH \rightarrow \kx$ are morphisms of algebraic groups.

\begin{lemma}  \label{torusPprime}  \cite[Lemma 4.3]{GY.pcluster}
Let $R$ be a noetherian Poisson algebra, equipped with a rational Poisson action of a torus $\HH$. Then any prime ideal of $R$ minimal over an $\HH$-stable Poisson ideal is itself an $\HH$-stable Poisson ideal, and the $\HH$-Poisson-prime ideals of $R$ are exactly the $\HH$-stable, Poisson, prime ideals of $R$.
\end{lemma}

Lemma \ref{torusPprime} allows a shortening of terminology: In the setting of the lemma,
\begin{equation}
\label{pHprime}
\HPspec R = \{ \text{Poisson} \; \HH\text{-prime ideals of} \; R\} = \Hspec R \cap \Pspec R \cap \Spec R.
\end{equation} 
This is immediate from the lemma together with the fact that the $\HH$-prime ideals of $R$ coincide with the $\HH$-stable prime ideals \cite[Proposition II.2.9]{BrGo}.

\begin{proposition}  \label{h.action}
Let $R$ be a $\KK$-algebra equipped with a rational action of a $\KK$-torus $\HH$, and set $\hfrak := \Lie \HH$.

{\rm(a)} The differential of the $\HH$-action gives an action of $\hfrak$ on $R$ by derivations. If $R$ is a Poisson algebra and the $\HH$-action is a Poisson action, then $\hfrak$ acts on $R$ by Poisson derivations.

{\rm(b)} The $\hfrak$-action on $R$ commutes with the $\HH$-action.

{\rm(c)} The $\hfrak$-eigenspaces in $R$ coincide with the $\HH$-eigenspaces.

{\rm(d)} Suppose $R'$ is another $\KK$-algebra, equipped with a rational action of the torus $\HH$ by $\KK$-algebra automorphisms. Then any $\HH$-equivariant $\KK$-algebra homomorphism $\phi : R \rightarrow R'$ is also $\hfrak$-equivariant.
\end{proposition}

\begin{proof} (a)(b)(c) See \cite[Section 1.2 and Lemmas 1.3, 1.4]{KGSL}.

(d) Since $\phi$ is $\HH$-equivariant, it is a homogeneous map with respect to the $X(\HH)$-gradings on $R$ and $R'$. It follows from \cite[Lemma 1.3]{KGSL} that $\phi(h \cdot r) = h\cdot \phi(r)$ for all $h \in \hfrak$ and all $X(\HH)$-homogeneous elements $r\in R$, and thus $\phi$ is $\hfrak$-equivariant.
\end{proof}

Whenever we have a $\KK$-algebra $R$ equipped with a rational action of a $\KK$-torus $\HH$, we assume that $\Lie \HH$ is acting on $R$ by the differential of the $\HH$-action.

%%%%%%%%%%%%%%%%%%%%%%%%%
\sectionnew{Poisson catenarity}

We address Poisson catenarity of general Poisson algebras in this section, beginning with an example for which Poisson catenarity fails. We then establish two general theorems. The first provides sufficient conditions for the Poisson prime spectrum of a Poisson algebra to be catenary, and the second provides a reduction of the key condition in the first theorem.

The following is a Poisson adaptation of an example of Bell and Sigur{\wasy\char107}sson \cite[Example 2.10]{BeSi}.

\begin{example}  \label{fromBellSig}
Let $R$ be the polynomial ring $\KK[x,y,z,w]$. There is a Poisson bracket on $R$ such that
\begin{align*}
\{x,y\} &= 0  &\qquad \{x,z\} &= 0  &\qquad \{y,z\} &= 0  \\
\{w,x\} &= 2yz  &\qquad \{w,y\} &= x+y^2  &\{w,z\} &= 0.
\end{align*}
The restriction of $\{w,-\}$ to $A := \KK[x,y,z]$ equals the derivation $\delta := 2yz\frac{\partial}{\partial x} + (x+y^2) \frac{\partial}{\partial y}$ and $R = A[w; 0,\de]_p$. 

By \cite[Example 2.10]{BeSi}, the $\delta$-invariant prime ideal $P := Ax+Ay$ of $A$ has $\delta$-height $1$, that is, $P$ does not properly contain any nonzero $\delta$-invariant prime ideals of $A$. We claim that the Poisson prime ideal $RP = Rx+Ry$ has height $1$ in $\Pspec R$, that is, $RP$ does not properly contain any nonzero Poisson prime ideals of $R$. 

Suppose, to the contrary, that $RP$ properly contains a nonzero Poisson prime ideal $Q$. Then $Q\cap A$ is a $\de$-invariant prime ideal of $A$ properly contained in $P$, whence $Q \cap A = 0$. Consider the polynomial ring $T := F[w]$ where $F := \Fract A$. As a localization of $R$, the ring $T$ is a Poisson algebra, and $TQ$ is a proper nonzero Poisson ideal of $T$. Note that the Poisson bracket vanishes on $F\times F$, since it vanishes on $A\times A$. Let
$$
t = t_0 + t_1 w + \cdots + t_n w^n
$$
be a nonzero polynomial in $TQ$ of minimal degree, and note that $n>0$. Since we may replace $t$ by $t_n^{-1}t$, there is no loss of generality in assuming that $t_n = 1$. Observe that
$$
\{x,t\} = n w^{n-1} \{x,w\} + [\text{lower terms}] = - 2n yz w^{n-1} + [\text{lower terms}],
$$
so that $\deg \{x,t\} = n-1$. Since $\{x,t\} \in TQ$, this contradicts the minimality of $n$.

Therefore $RP$ has height $1$ in $\Pspec R$, as claimed. Consequently, the chain
$$
0 < Rx+Ry < Rx+Ry+Rz
$$
in $\Pspec R$ is saturated. However, there is also a saturated chain
$$
0 < Rz < Rx+Rz < Rx+Ry+Rz
$$
in $\Pspec R$. Therefore $\Pspec R$ is not catenary.
\end{example}

There are two general theorems in the literature that have been used to establish catenarity for the prime spectra of various quantum algebras -- \cite[Theorem 1.6]{KGTL} and \cite[Theorem 0.1]{YZ}. A key hypothesis in both theorems is \emph{normal separation}, a condition that requires a suitable supply of normal elements in prime factor rings. An analogous condition, as follows, is important in the Poisson setting.

\begin{definition}  \label{P.normal}
Let $R$ be a Poisson algebra.

An element $c\in R$ is called \emph{Poisson-normal} provided $\{c,R\} \subseteq Rc$, that is, $Rc$ is a Poisson ideal of $R$.

The set $\Pspec R$ is said to have \emph{Poisson-normal separation} provided that for any pair of distinct comparable prime ideals $P \subsetneq Q$ in $R$, the factor $Q/P$ contains a nonzero Poisson-normal element of $R/P$.
\end{definition} 

\begin{theorem}  \label{saturatedinPspecvsSpec}
Let $R$ be a noetherian Poisson algebra such that $\Pspec R$ has Poisson-normal separation. Then any saturated chain in $\Pspec R$ is also saturated as a chain in $\Spec R$.
\end{theorem}

\begin{proof} We just need to show that if $P \subsetneq Q$ are Poisson prime ideals such that no Poisson prime ideals lie strictly between $P$ and $Q$, then no prime ideals lie strictly between $P$ and $Q$, that is, $\hgt(Q/P) = 1$.

By Poisson-normal separation, there exists $c\in Q \setminus P$ such that $c$ is Poisson-normal modulo $P$. Hence, $Rc+P$ is a Poisson ideal contained in $Q$. Let $Q' \subseteq Q$ be a prime ideal minimal over $Rc+P$. By Lemma \ref{Pprimeinfo}, $Q'$ is a Poisson ideal, and our assumption on $P \subsetneq Q$ implies that $Q'=Q$. Thus $Q$ is minimal over $Rc+P$. Krull's Principal Ideal Theorem therefore implies $\hgt(Q/P)=1$, as required.
\end{proof}

We refer to a Poisson algebra as \emph{affine} when the underlying associative algebra is affine (i.e., finitely generated as an algebra).

\begin{corollary}  \label{affine+Pnormsep>Pcat}
Let $R$ be an affine Poisson algebra. If $\Pspec R$ has Poisson-normal separation, then $\Pspec R$ is catenary.
\end{corollary}

\begin{proof} Since $R$ is an affine commutative algebra, $\Spec R$ is catenary (e.g., \cite[Corollary 13.6]{Eis}). Given $P \subsetneq Q$ in $\Pspec R$, any two saturated chains of Poisson prime ideals from $P$ to $Q$ are also saturated as chains of ordinary prime ideals by Theorem \ref{saturatedinPspecvsSpec}, so they have the same length due to catenarity of $\Spec R$.
\end{proof}

\begin{definition}  \label{Pheight}
The \emph{Poisson height} of a Poisson prime ideal $P$ in a Poisson algebra $R$, denoted $\pht P$, is the height of $P$ within the poset $\Pspec R$, that is, the supremum of the lengths of all chains of Poisson prime ideals in $R$ descending from $P$.
\end{definition} 

\begin{corollary}  \label{Pht=ht}
Let $R$ be an affine Poisson algebra. If $\Pspec R$ has Poisson-normal separation, then $\pht P = \hgt P$ for all $P \in \Pspec R$.
\end{corollary}

\begin{proof} Let $P \in \Pspec R$. Any minimal prime ideal $Q$ contained in $P$ is a Poisson prime by Lemma \ref{Pprimeinfo}(a), and any saturated chain from $P$ to $Q$ in $\Pspec R$ is also saturated in $\Spec R$ by Theorem \ref{saturatedinPspecvsSpec}. Since $\Pspec R$ and $\Spec R$ are both catenary, it follows that $\pht P/Q = \hgt P/Q$. Taking suprema over all such $Q$ yields $\pht P = \hgt P$.
\end{proof}

\begin{remark}  
In many quantum algebras $T$, catenarity of $\Spec T$ is accompanied by \emph{Tauvel's height formula}, which says that
$$
\hgt P + \GK T/P = \GK T
$$
for all prime ideals $P$ of $T$, assuming that $T$ is a prime ring.  An analogous Poisson version for an affine Poisson domain $R$ would say that
\begin{equation}  \label{P-Tauvel}
\pht P + \GK R/P = \GK R
\end{equation}
for all Poisson prime ideals $P$ of $R$. If $\Pspec R$ has Poisson-normal separation, then \eqref{P-Tauvel} follows from Corollary \ref{Pht=ht}, since $\GK = \Kdim$ for commutative affine algebras (e.g., \cite[Theorem 4.5]{KrLe}) and $\hgt P+ \Kdim R/P = \Kdim R$ for prime ideals $P$ in such algebras $R$ (e.g., \cite[Corollary 13.4]{Eis}).

A more ``fully Poisson" version of \eqref{P-Tauvel} would replace GK-dimension with \emph{Poisson GK-dimension} as defined in \cite[Definition 3.5]{PeSi}. However, the resulting equations would be the same as \eqref{P-Tauvel}, since, as is easily verified, Poisson GK-dimension agrees with GK-dimension on affine Poisson algebras.
\end{remark}

Our approach to the second main theorem of this section involves a stratification of the Poisson prime spectrum of a Poisson algebra.

\begin{definition}  \label{Hcore}
Let $R$ be a ring and $\HH$ a group acting on $R$ by ring automorphisms.
The \emph{$\HH$-core} of an ideal $I$ of $R$ is the ideal $(I:\HH) := \bigcap_{h\in\HH} h\cdot I$. That is, $(I:\HH)$ is the largest $\HH$-stable ideal of $R$ contained in $I$.
\end{definition}

\begin{lemma}  \label{H.Poisson.ideals}  \cite[Lemma 3.1]{GpDixMo}
Let $R$ be a noetherian Poisson algebra equipped with a rational Poisson action of an algebraic torus $\HH$.

{\rm(a)} If $I$ is a Poisson ideal of $R$, then $(I:\HH)$ is an $\HH$-stable Poisson ideal. 

{\rm(b)} $(P:\HH) \in \HPspec R$ for all $P \in \Pspec R$.
\end{lemma}

\begin{definition}  \label{HPstrat}
Let $R$ be a noetherian Poisson algebra equipped with a rational Poisson action of an algebraic torus $\HH$. For $J \in \HPspec R$, set
$$
\Pspec_J R = \{ P \in \Pspec R \mid (P:\HH) = J \}.
$$
In view of Lemma \ref{H.Poisson.ideals}(b), we obtain a partition
$$
\Pspec R = \bigsqcup_{J \in \HPspec R} \Pspec_J R,
$$
which is known as the \emph{$\HH$-stratification of $\Pspec R$}.

Let $\E_J$ denote the set of $\HH$-eigenvectors in $R/J$. Since $R/J$ is a domain (Lemma \ref{torusPprime}), $\E_J$ is multiplicatively closed, and the localization $R_J := (R/J)[\E_J^{-1}]$ is a subalgebra of the quotient field $\Fract R/J$. The Poisson bracket on $R$ uniquely induces Poisson brackets on $R/J$ and $R_J$ (as well as on $\Fract R/J$), so that these algebras become Poisson algebras. Similarly, the action of $\HH$ uniquely induces actions on $R/J$ and $R_J$ by Poisson automorphisms, as well as on $\Fract R/J$, although the latter action is no longer rational. 
\end{definition}

\begin{theorem}  \label{Pstratif} \cite[Theorem 4.2]{GpDixMo}
Let $R$ be a noetherian Poisson algebra equipped with a rational Poisson action of an algebraic torus $\HH$. Let $J \in \HPspec R$.

{\rm (a)} The algebra $R_J$ is a graded field with respect to its induced
$X(\HH)$-grading, i.e., all nonzero homogeneous elements of $R_J$ are invertible.

{\rm (b)} $\Pspec_J R$ is homeomorphic to $\Pspec R_J$ via localization and
contraction.

{\rm (c)} $\Pspec R_J$ is homeomorphic to $\Spec Z_p(R_J)$ via contraction
and extension.
\end{theorem}

From this theorem we obtain a Poisson analog of the first conclusion of \cite[Theorem 5.3]{murcia.proc} as follows.

\begin{definition}  \label{PnormalHsep}
Suppose $R$ is a Poisson algebra equipped with a Poisson action of a group $\HH$. Adapting the notation of \cite[\S5.2]{murcia.proc}, we say that $\HPspec R$ has \emph{Poisson-normal $\HH$-separation} provided that for any proper inclusion $P \subsetneq Q$ of Poisson $\HH$-prime ideals of $R$, there exists a Poisson-normal $\HH$-eigenvector of $R/P$ which lies in $Q/P$.
\end{definition}

\begin{theorem}  \label{HPnormsep>Pnormsep}
Let $R$ be a noetherian Poisson algebra equipped with a rational Poisson action of an algebraic torus $\HH$. If $\HPspec R$ has Poisson-normal $\HH$-separation, then $\Pspec R$ has Poisson-normal separation.
\end{theorem}

\begin{proof}
Let $P \subsetneq Q$ be a proper inclusion of Poisson prime ideals of $R$. Set $J := (P:\HH)$ and $K := (Q:\HH)$, which are members of $\HPspec R$ such that $P \in \Pspec_J R$ and $Q \in \Pspec_K R$. Obviously $J \subseteq K$.  

If $J \ne K$, then by Poisson-normal $\HH$-separation there is an $\HH$-eigenvector $c\in K \setminus J$ which is Poisson-normal modulo $J$. Obviously $c\in Q$ and $c$ is Poisson-normal modulo $P$. Since $Rc$ is an $\HH$-invariant ideal not contained in $J$, we see that $Rc \nsubseteq P$, that is, $c\notin P$. Thus in this case we are done.

Now assume that $J = K$. Since it is harmless to pass to $R/J$, there is no loss of generality in assuming that $J=K=0$. We shall need the following fact:
\begin{equation}  \label{R0denoms}
Z_p(R_0) \subseteq \{ ac^{-1} \mid a,c \in R\ \text{and}\ c\ \text{is a Poisson-normal}\ \HH\text{-eigenvector} \}.
\end{equation}
First note that the induced $\HH$-action on $R_0$ is rational, equivalently, $R_0$ is $\xh$-graded. Since $\HH$ acts by Poisson automorphisms, the Poisson bracket on $R_0$ is $X(\HH)$-graded, that is, $\{ (R_0)_u, (R_0)_v \} \subseteq (R_0)_{u+v}$ for all $u,v \in X(\HH)$. It follows that an element $z\in R_0$ is Poisson-central if and only if all its homogeneous components are Poisson-central. Hence, to prove \eqref{R0denoms} it suffices to show that any homogeneous element $z\in Z_p(R_0)$ can be written in the desired form.

Given such a $z$, consider the nonzero ideal $I := \{ r\in R \mid zr \in R \}$. Since $z$ is both homogeneous and Poisson-central, $I$ is an $\HH$-stable Poisson ideal. If $z\in R$, we can write $z = z1^{-1}$ in the desired form, so we may assume $z\notin R$, whence $I \ne R$. There are prime ideals $P_1,\dots,P_m$ minimal over $I$ such that $P_1P_2\cdots P_m \subseteq I$. Lemma \ref{torusPprime} implies that the $P_i$ are Poisson $\HH$-prime ideals. By Poisson-normal $\HH$-separation, each $P_i$ contains a Poisson-normal $\HH$-eigenvector $c_i$. Hence, $c := c_1c_2 \cdots c_m$ is a Poisson-normal $\HH$-eigenvector lying in $I$, and $z = ac^{-1}$ where $a = zc \in R$. This establishes \eqref{R0denoms}.

Since $J = K = 0$, we have $P,Q \in \Pspec_0 R$. By Theorem \ref{Pstratif}(b)(c), $R_0P \cap Z_p(R_0) \subsetneq R_0Q \cap Z_p(R_0)$, so there exists a Poisson-central element $z \in R_0Q \setminus R_0P$. Apply \eqref{R0denoms} to write $z = ac^{-1}$ for some $a,c\in R$ with $c$ a Poisson-normal $\HH$-eigenvector of $R$. Note that $a\in Q \setminus P$, since $R_0Q \cap R = Q$ and $R_0P \cap R = P$. Given any $r\in R$, we have $\{c,r\} = cs$ for some $s\in R$, whence
$$
\{a,r\} = \{zc,r\} = z\{c,r\} = zcs = as.
$$
Therefore $a$ is Poisson-normal in $R$, hence also Poisson-normal modulo $P$, and the proof is complete
\end{proof}

%%%%%%%%%%%%%%%%%%%%%%%%%
\sectionnew{A construction of Poisson-normal elements}  \label{construct}

\subsection{Basic assumptions}  \label{pCauchon.assump} 
Assume throughout this section that $R = A[X; \sig, \de]_p$ is a \emph{Poisson-Cauchon extension} in the sense of \cite[Definition 4.4]{GY.pcluster}, that is, the Poisson polynomial algebra $R$ is equipped with a rational Poisson action of a $\KK$-torus $\HH$ such that
\begin{itemize}
\item The subalgebra $A$ is $\HH$-stable, and $X$ is an $\HH$-eigenvector.
\item $\de$ is locally nilpotent.
\item There exists $h_\circ \in \hfrak := \Lie \HH$ such that $(h_\circ \cdot)|_A = \sig$ and the $h_\circ$-eigenvalue $\la_\circ$ of $X$ is nonzero.
\end{itemize}
The third condition relies on $X$ being an $\hfrak$-eigenvector, which follows from Proposition \ref{h.action}(c).

We write $\chi_r : \HH \rightarrow \kx$ to denote the $\HH$-eigenvalue of an $\HH$-eigenvector $r\in R$.

\subsection{The Poisson Cauchon map}
Set $\Rhat := A[X^{\pm1}; \sig, \de]_p$. Since $\Rhat$ is a localization of $R$ at a multiplicative set consisting of $\HH$-eigenvectors, the action of $\HH$ on $R$ extends uniquely to a rational action on $\Rhat$ by $\KK$-algebra automorphisms. As is easily checked, this is a Poisson action of $\HH$ on $\Rhat$.
Let $\theta : A \rightarrow \Rhat$ be the \emph{Poisson Cauchon map} defined by
\begin{equation}  \label{pCauchonmap}
\theta(a) = \sum_{l=0}^\infty \frac1{l!} \left( \frac{-1}{\la_\circ} \right)^l \de^l(a) X^{-l}.
\end{equation} 

\begin{proposition}  \label{theta.properties} \cite[Lemmas 3.4, 3.6, Theorem 3.7]{KGSL}; \cite[Proposition 4.5, Corollary 4.6]{GY.pcluster}

{\rm(a)} $\theta$ is an injective Poisson algebra homomorphism.

{\rm(b)} $\{X, \theta(a) \} = \theta\sigma(a) X$ for all $a\in A$.

{\rm(c)} $\theta$ extends uniquely to an injective Poisson algebra homomorphism $A[Y;\sig]_p \rightarrow \Rhat$ with $\theta(Y) = X$, and then to a Poisson algebra isomorphism $A[Y^{\pm1};\sig]_p \rightarrow \Rhat$.

{\rm(d)} Set $B := \theta(A)$ and $T := \theta(A[Y;\sig]_p) \subseteq \Rhat$. Then $T = B[X;\al]_p$ where $\al$ is the Poisson derivation of $B$ defined by $\al(\theta(a)) = \theta(\sig(a))$.

{\rm(e)} The localization $T[X^{-1}] = B[X^{\pm1};\al]_p$ coincides with $\Rhat$.

{\rm(f)} The $\HH$-action on $A$ extends uniquely to a rational Poisson action of $\HH$ on $A[Y^{\pm1};\sig]_p$ such that $\chi_Y = \chi_X$. With respect to this action, the isomorphism $\theta : A[Y^{\pm1};\sig]_p \rightarrow \Rhat$ is $\HH$-equivariant.

{\rm(g)} $\al = (h_\circ\cdot)|_B$.
\end{proposition}

\begin{lemma}  \label{max.des(a)}
Let $a\in A$ be nonzero and let $s \in \Znn$ be maximal such that $\de^s(a) \ne 0$. Then $s$ is minimal such that $\theta(a)X^s \in R$.
\end{lemma}

\begin{proof} Since $\de^l(a) = 0$ for $l>s$, we have $\theta(a) = \sum_{l=0}^s c_l \de^l(a) X^{-l}$ for some nonzero $c_l \in \KK$. Thus $\theta(a)X^s \in R$.

Suppose that $s>0$ and $\theta(a)X^t \in R$ for some $t<s$. Then $\theta(a)X^{s-1} \in R$, whence $\de^s(a) X^{-1} \in R$. Consequently, $\de^s(a) \in A \cap RX = 0$, which contradicts our choice of $s$. Therefore $s$ is minimal such that $\theta(a)X^s \in R$.
\end{proof}

The following Poisson version of \cite[Lemma 2.2]{KGSL2} is adapted from the proof of \cite[Theorem 4.7]{GY.pcluster}.

\begin{proposition}  \label{firstPnormal}
Assume that $A$ is a domain.
Let $a\in A$ be a Poisson-normal $\HH$-eigenvector, and let $s \in \Znn$ be maximal such that $\de^s(a) \ne 0$. Then the element $x := \theta(a)X^s$ is a Poisson-normal $\HH$-eigen\-vector in $R$. In particular, $\{ x, X \} = - \eta x X$, where $\eta$ is the $\sig$-eigenvalue of $a$.
\end{proposition} 

\begin{proof} In view of Proposition \ref{theta.properties}, the element $b := \theta(a)$ is a Poisson-normal $\HH$-eigenvector in $B$, and the $h_\circ$-eigenvalue of $b$ equals that of $a$ (Proposition \ref{h.action}(d)), namely $\eta$. By Lemma \ref{max.des(a)}, $s$ is minimal such that $bX^s \in R$. This places $x \in R$, and clearly $x$ is an $\HH$-eigenvector.

Since
\begin{equation}  \label{X,b}
\{ X, b \} = \al(b) X = (h_\circ\cdot b) X = \eta b X,
\end{equation}
we see that $\{ x, X \} = - \eta x X$. It follows also that $b$ is Poisson-normal in $T$ and in $\Rhat$. In particular, $\Rhat b$ is a Poisson ideal of $\Rhat$. Since $\Rhat b = \Rhat x$, the ideal
$$
I := \Rhat x \cap R
$$
is a Poisson ideal of $R$. We show that $I = Rx$, which will prove that $x$ is Poisson-normal in $R$. Obviously $I \supseteq Rx$.

Let $y \in I$. Then $y \in \Rhat b$ implies $y X^u \in Tb$ for some $u \ge 0$. Now $y X^u = cb$ for some $c\in T$, and $cX^v \in R$ for some $v \ge 0$. Then
$$
y X^{u+v+s} = cb X^{v+s} = c X^v x \in Rx.
$$
Let $t \in \Znn$ be minimal such that $y X^t \in Rx$, and write $y X^t = rx$ for some $r\in R$.

We will show that $t=0$. Write
$$
r = \sum_{i\ge0} r_i X^i, \qquad y = \sum_{i\ge0} y_i X^i, \qquad x = \sum_{i\ge0} x_i X^i
$$
for some $r_i,y_i,x_i \in A$. In case $s=0$, we would have $x= b = a \in A$ and so $x_0 = a \ne 0$. In case $s>0$, we would have
$$
x_0 X^{-1} + \sum_{i\ge1} x_i X^{i-1} = x X^{-1} = b X^{s-1} \notin R
$$
by the minimality of $s$, again yielding $x_0 \ne 0$. Thus, $x_0 \ne 0$ in any case.

Observe that
$$
\sum_{i\ge0} y_i X^{i+t} = y X^t = rx = \sum_{i\ge0} r_i x X^i = \sum_{i,j\ge0} r_i x_j X^{i+j}.
$$
If $t>0$, it would follow that $r_0 x_0 = 0$, whence $r_0 = 0$. But then $r = r'X$ for some $r' \in R$, and so 
$$
y X^{t-1} = rx X^{-1} = r' x \in Rx,
$$
contradicting the minimality of $t$. Therefore $t=0$.

Consequently, $y = rx$, proving that $I = Rx$ as desired.
\end{proof}

%%%%%%%%%%%%%%%%%%%%%%%%%
\sectionnew{Poisson-normal elements in Poisson-Cauchon extensions}
Throughout this section, keep the assumptions of Section \ref{construct}, so that $R = A[X;\sig,\de]_p$ is a Poisson-Cauchon extension. Assume in addition that $A$ is a noetherian domain.

\subsection{$\HH$-primes in Poisson-Cauchon extensions}

\begin{lemma}  \label{PHcontract}
{\rm(a)} Every Poisson $\HH$-prime ideal of $R$ contracts to a $\de$-stable Poisson $\HH$-prime ideal of $A$.

{\rm(b)} For any $\de$-stable Poisson $\HH$-prime ideal $P_0$ of $A$, there are at most two Poisson $\HH$-prime ideals of $R$ that contract to $P_0$ in $A$. There is always at least one, namely $RP_0$.
\end{lemma}

\begin{proof} Recall that any $\HH$-stable ideal of $R$ or $A$ is also $\hfrak$-stable.

(a) If $Q$ is a Poisson $\HH$-prime ideal of $R$, then clearly $Q \cap A$ is a Poisson $\HH$-prime ideal of $A$. Given $a\in Q\cap A$, we have $\{ X,a \} \in \{X, Q \} \subseteq Q$ and $\sig(a) = h_\circ \cdot a \in h_\circ \cdot Q \subseteq Q$, whence
$$
\de(a) = \{ X,a \} - \sig(a) X \in Q
$$
and thus $\de(a) \in Q\cap A$.

(b) Since the action of $\hfrak$ on $R$ is by Poisson derivations (Proposition \ref{h.action}(a)), the action of $(h_\circ\cdot)$ on $R$ is, in particular, a derivation. Moreover, by assumption $(h_\circ\cdot)|_A = \sig$ and $h_\circ \cdot X = \la_\circ X$ with $\la_\circ \in \kx$. Thus, \cite[Proposition 1.2]{KGSL} implies that there are at most two $(h_\circ\cdot)$-stable Poisson prime ideals of $R$ that contract to $P_0$. Consequently, there are at most two Poisson $\HH$-prime ideals of $R$ that contract to $P_0$ in $A$.

Since $P_0$ is stable under both $\sig = (h_\circ \cdot)$ and $\de$, the induced ideal $RP_0$ is stable under $\{X,-\}$, and hence $RP_0$ is a Poisson ideal of $R$. It is clearly also an $\HH$-prime ideal, and it contracts to $P_0$ in $A$.
\end{proof}

We shall need 
\begin{observation}  \label{obs.minHprime}
It is immediate from the definition that the $\HH$-core of a prime ideal of $A$ (or $R$) is an $\HH$-prime ideal. Invoking \cite[Proposition II.2.9]{BrGo}, we conclude that
\begin{enumerate}
\item[(i)] If $P$ is a prime ideal of $A$ (or $R$), then its $\HH$-core $(P:\HH)$ is a prime and $\HH$-prime ideal of $A$ (or $R$).
\end{enumerate}
It follows from (i) that
\begin{enumerate}
\item[(ii)] If $I$ is an $\HH$-stable ideal of $A$ (or $R$), then all prime ideals minimal over $I$ are $\HH$-prime.
\end{enumerate}
\end{observation}

Set $A^* := \Fract A$ and $R^* := A^*[X]$. Since $A^*$ and $R^*$ are localizations of $A$ and $R$, the Poisson structures on $A$ and $R$ extends uniquely to Poisson structure on $A^*$ and $R^*$. As is easily checked, $A^*$ is a Poisson subalgebra of $R^*$, and $\{ X, A^* \} \subseteq A^* + A^* X$. Hence, $R^*$ is a Poisson polynomial ring $A^*[X; \sig^*, \de^*]_p$ for suitable $\sig^*$, $\de^*$. In particular, $\sig^*$ and $\de^*$ are derivations on $A^*$ extending $\sig$ and $\de$, so they are the unique extensions determined by the quotient rule. We thus label these extensions $\sig$ and $\de$ as well, and write $R^* = A^*[X; \sig, \de]_p$.

The actions of $\HH$ on $A$ and $R$ extend uniquely to Poisson actions on $A^*$ and $R^*$. However, these actions are not rational (unless $\HH$ acts trivially), so we cannot differentiate them to obtain $\hfrak$-actions. On the other hand, the $\hfrak$-actions on $A$ and $R$ do extend (uniquely) to actions on $A^*$ and $R^*$ by derivations, via the quotient rule. It can be checked that the actions of $\hfrak$ on $A^*$ and $R^*$ are by Poisson derivations, but we do not require these actions here.

Let us say that the algebra $R^*$ is \emph{$\HH$-Poisson simple} in case the only $\HH$-stable Poisson ideals of $R^*$ are $0$ and $R^*$. Thus, in view of Lemma \ref{torusPprime}, $R^*$ is $\HH$-Poisson simple if and only if $0$ is the unique Poisson $\HH$-prime ideal of $R$.

\begin{proposition}  \label{R*notPHsimple}
Suppose that $R^*$ is not $\HH$-Poisson simple.

{\rm(a)} There is a unique element $d\in A^*$ such that
\begin{enumerate}
\item[\rm(i)] $d$ is $X(\HH)$-homogeneous with $\al\cdot d = \chi_X(\al) d$ for all $\al \in \HH$.
\item[\rm(ii)] $\{ d,a \} = \sig(a) d + \de(a)$ for all $a\in A^*$.
\end{enumerate}
 In particular, $X-d$ is an $\HH$-eigenvector with the same $\HH$-eigenvalue as $X$.
 
 {\rm(b)} $\sig(d) = \la_\circ d$ and $\de(d) = - \la_\circ d^2$.

{\rm(c)} There is a unique nonzero Poisson $\HH$-prime ideal in $R^*$, namely $R^*(X-d)$.

{\rm(d)} Let $I^*$ be a proper nonzero $\HH$-stable Poisson ideal of $R^*$, let $n$ be the minimum degree for nonzero elements of $I^*$, and let $f = X^n + cX^{n-1} + \text{\rm[lower terms]}$, with $c\in A^*$, be a monic element of $I^*$ with degree $n$. Then $n>0$ and $d = (-1/n) c$.
\end{proposition}

\begin{proof} Since $R^*$ is not $\HH$-Poisson simple, it contains a proper nonzero $\HH$-stable Poisson ideal $I^*$. Let $n$, $f$, and $c$ be as in part (d). For any $a\in A^*$, the ideal $I^*$ contains the polynomial
$$
\{ f,a \} - n \sig(a) f = [ n \de(a) + \{ c,a \} + c (n-1) \sig(a) - n \sig(a) c] X^{n-1} + \text{[lower terms]},
$$
which must vanish due to the minimality of $n$, and so $n \de(a) + \{ c,a \} = c \sig(a)$. Hence, the element
$$
d := \frac{-1}n c
$$
satisfies condition (a)(ii). 

Set $Y := X - d$. Then $\{ Y,a \} = \sig(a) Y$ for $a\in A^*$, so $R^*$ is a Poisson polynomial algebra of the form $A^*[Y;\sig]_p$. In particular, $Y$ is Poisson-normal in $R^*$, so that $R^*Y$ is a Poisson ideal of $R^*$. It is also a prime ideal.

Returning to the polynomial $f$, note that 
\begin{align*}
(\al \cdot f) - \chi_X(\al)^n f &= [ (\al \cdot c) \chi_X(\al)^{n-1} - \chi_X(\al)^n c ] X^{n-1} + \text{[lower terms]} \qquad \forall\, \al \in \HH  \\
(h_\circ \cdot f) - n \la_\circ f &= [ (h_\circ \cdot c) + (n-1) c \la_\circ - n \la_\circ c ] X^{n-1} + \text{[lower terms]}.
\end{align*}
Since these polynomials lie in $I^*$, the minimality of $n$ now implies that $(\al \cdot c) = \chi_X(\al) c$ for $\al \in \HH$ and $h_\circ \cdot c = \la_\circ c$. Consequently, condition (a)(i) holds for $d$ and $\sig(d) = \la_\circ d$. Thus, $Y$ is an $\HH$-eigenvector with the same $\HH$-eigenvalue as $X$. In particular, it follows that $R^*Y$ is $\HH$-stable. Therefore $R^*Y$ is a Poisson $\HH$-prime ideal of $R^*$. Further, $\de(d) = \{d,d\} - \sig(d)d = - \la_\circ d^2$, verifying part (b).

It only remains to prove the uniqueness statements in parts (a) and (c). Since $A^*$ is a field, any nonzero Poisson $\HH$-prime ideal $Q^*$ of $R^*$ contracts to $0$ in $A^*$. Hence, $Q^* \cap R$ and $R^*Y \cap R$ are nonzero Poisson $\HH$-prime ideals of $R$ that contract to $0$ in $A$. Lemma \ref{PHcontract}(b) implies that $Q^* \cap R = R^*Y \cap R$, and so $Q^* = R^*Y$. This establishes part (c).

Finally, suppose that $d'$ is an element of $A^*$ satisfying conditions (i),(ii). If $Y' := X-d'$, we observe as above that $R^*Y'$ is a Poisson $\HH$-prime ideal of $R^*$. By the uniqueness result of part (c), $R^*Y' = R^*Y$. Since $Y'$ and $Y$ are monic of degree $1$, they must be equal, whence $d'=d$. This completes the proof of part (a).
\end{proof} 

Whenever $R^*$ is not $\HH$-Poisson simple, we keep the notation $d$ for the unique element of $A^*$ described in Proposition \ref{R*notPHsimple}(a). Note that $R^* = A^*[X-d; \sig]_p$ in this case, and that the analogs of Lemma \ref{PHcontract}(a) and Observation \ref{obs.minHprime}(i),(ii) hold for $R^*$ and $A^*$.

\begin{corollary}  \label{R*notPHsimple2}
If $R^*$ is not $\HH$-Poisson simple, then $R^*(X-d) \cap R$ is the unique nonzero Poisson $\HH$-prime ideal of $R$ that contracts to $0$ in $A$. Moreover, any $\HH$-stable Poisson ideal of $R$ that contracts to $0$ in $A$ is contained in $R^*(X-d) \cap R$.
\end{corollary}

\begin{proof} On one hand, $P^* := R^*(X-d)$ is a nonzero Poisson $\HH$-prime ideal of $R^*$ that contracts to zero in $A^*$, whence $P^*\cap R$ is a nonzero Poisson $\HH$-prime ideal of $R$ that contracts to $0$ in $A$. On the other hand, any nonzero Poisson $\HH$-prime ideal $Q$ of $R$ with $Q\cap A = 0$ localizes to a nonzero Poisson $\HH$-prime ideal $R^*Q$ in $R^*$, whence $R^*Q = P^*$ and thus $Q = R^*Q\cap R = P^* \cap R$.

Similarly, any $\HH$-stable Poisson ideal $I$ of $R$ with $I\cap R =0$ localizes to an $\HH$-stable Poisson ideal $R^*I$ of $R^*$. Since $I\cap R =0$, we must have $R^*I \ne R^*$, whence there is at least one prime ideal $Q^*$ of $R^*$ minimal over $R^*I$. Then $Q^*$ is a Poisson $\HH$-prime ideal by Lemma \ref{Pprimeinfo}(a) and Observation \ref{obs.minHprime}(ii), whence $Q^* = P^*$. Therefore $I \subseteq R^*I \cap R \subseteq P^*\cap R$.
\end{proof}

\subsection{Some Poisson-normal $\HH$-eigenvectors}

We upgrade \cite[Lemma 4.11(a)]{GY.pcluster} in the same manner as \cite[Lemma 2.3]{KGSL2}:

\begin{lemma}  \label{Richard.upgrade}
Let $\de$ be a derivation on a domain $C$ of characteristic zero, and suppose $e,f \in C$ with $\de(e) = ef$ or $\de(e) = fe$. If there is some $m \in \Znn$ such that $\de^m(e) = \de^m(f) = 0$, then $\de(e) = 0$.
\end{lemma}

\begin{proof} It suffices to show that one of $e$ or $f$ is zero. Suppose not, and let $s,t \in \Znn$ be maximal such that $\de^s(e), \de^t(f) \ne 0$.  By Leibniz' Rule,
$$
\de^{s+t}(ef) = \sum_{i=1}^{s+t} \tbinom{s+t}{i} \de^i(e) \de^{s+t-i}(f) = \tbinom{s+t}{s} \de^s(e) \de^t(f) \ne 0,
$$
since $\binom{s+t}{s} \ne 0$ in characteristic zero. Similarly, $\de^{s+t}(fe) \ne 0$. But then $\de^{s+t+1}(e) \ne 0$, due to the assumption that $\de(e) = ef$ or $\de(e) = fe$. This contradicts the choice of $s$, since $s+t+1 > s$.
\end{proof} 

\begin{lemma}  \label{PnormalHeigen}
Assume there is a nonzero Poisson $\HH$-prime ideal $P$ in $R$ with $P\cap A =0$. Let $a\in A$ be a Poisson-normal $\HH$-eigenvector, and let $s\in \Znn$ be maximal such that $\de^s(a) \ne 0$.

{\rm(a)} If $s>0$, then $x := \theta(a) X^s$ is a Poisson-normal $\HH$-eigenvector in $R$ and $x\in P$. Moreover, $d = (\la_\circ s)^{-1} a^{-1} \de(a)$.

{\rm(b)} Now assume that $a$ is the leading coefficient of some element of $P$ with degree $1$. Then $a+P$ is a Poisson-normal $\HH$-eigenvector in $R/P$. Moreover, if also $s=0$, then $\de \equiv 0$ and $P= RX$.
\end{lemma}

\begin{proof} Note that $a \notin P$, because $P \cap A = 0$. The ideal $P$ localizes to a nonzero Poisson $\HH$-prime ideal $P^*$ of $R^*$ such that $P^*\cap R = P$, and $P^* = R^*(X-d)$ by Proposition \ref{R*notPHsimple}(c).

(a) By Proposition \ref{firstPnormal}, $x$ is a Poisson-normal $\HH$-eigenvector in $R$. Now $Rx$ is a nonzero $\HH$-stable Poisson ideal of $R$, and $Rx \cap A = 0$ because $\deg x = s > 0$. By Corollary \ref{R*notPHsimple2}, $Rx \subseteq P$, whence $x\in P$.

Note that $x = a X^s + c X^{s-1} + \text{[lower terms]}$, where $c = - \la_\circ^{-1} \de(a)$. The ideal $Rx$ localizes to a proper nonzero $\HH$-stable Poisson ideal $R^*x$ in $R^*$, and $s$ is the minimum degree for nonzero elements of $R^*x$. Since $a^{-1} x$ is a monic element of $R^*x$ with degree $s$, Proposition \ref{R*notPHsimple}(d) implies that $d = (-1/s) a^{-1} c = (\la_\circ s)^{-1} a^{-1} \de(a)$.

(b) By hypothesis, $aX + c \in P$ for some $c\in A$. Then $X + a^{-1}c$ is a monic element of $P^*$ with degree $1$. Since $P^*$ is proper, it contains no nonzero elements of degree $0$. Hence, we again apply Proposition \ref{R*notPHsimple}(d), obtaining $d = - a^{-1} c$. 

If $s=0$, then $\de(a) = 0$, whence $\de^m(d) = - a^{-1} \de^m(c) = 0$ for some $m \in \Znn$. Since $\de(d) = - \la_\circ d^2$ (Proposition \ref{R*notPHsimple}(b)), Lemma \ref{Richard.upgrade} implies that $\de(d) = 0$. It follows that $d = 0$, whence $\de \equiv 0$ in this case. We then have $P = R^*X \cap R = RX$. Moreover, $a=x$ and so
$$
\{ a, X \} = - \eta a X \in Ra,
$$
where $\eta$ is the $\sig$-eigenvalue of $a$ (Proposition \ref{firstPnormal}). Thus, $a$ is Poisson-normal in $R$, whence $a+P$ is Poison-normal in $R/P$.

Finally, assume that $s>0$. By part (a), we have
$$
a^{-1} c = - d = - (\la_\circ s)^{-1} a^{-1} \de(a),
$$
whence $\de(a) = - \la_\circ s c$. Since $aX+c \in P$, it follows that
$$
\{ X,a \} = \eta a X - \la_\circ s c \equiv \eta a X - \la_\circ s (-aX) = (\eta + \la_\circ s) X a \pmod{P}.
$$
As $a$ is already Poisson-normal in $A$, we conclude that $a+P$ is Poisson-normal in $R/P$.
\end{proof}

\begin{proposition}  \label{Pnormalsep1}
Assume that every nonzero Poisson $\HH$-prime ideal of $A$ contains a Poisson-normal $\HH$-eigenvector.

If $P \subsetneq Q$ are Poisson $\HH$-prime ideals of $R$ with $P \cap A = 0$, then there exists a Poisson-normal $\HH$-eigenvector $u$ of $R/P$ such that $u \in Q/P$.
\end{proposition}

\begin{proof} Recall that $Q\cap A$ is a $\de$-stable Poisson $\HH$-prime ideal of $A$.

Assume first that $P \ne 0$. Then $0$ and $P$ are two Poisson $\HH$-prime ideals of $R$ that contract to $0$ in $A$, so $Q \cap A \ne 0$ by Lemma \ref{PHcontract}(b).

Now $P$ localizes to a nonzero Poisson $\HH$-prime ideal $P^*$ in $R^*$, and $P^* = R^*(X-d)$ by Proposition \ref{R*notPHsimple}(c). Writing $d = bc^{-1}$ for some $b,c\in A$ with $c\ne 0$, we have $cX - b = c(X-d) \in P^*\cap R = P$. Thus, the ideal
$$
J := \{ a\in A \mid aX+e \in P\ \text{for some}\ e \in A \}
$$
is nonzero, as is then $J \cap (Q \cap A) = J \cap Q$. Note that $J \cap Q \ne A$, since $Q \ne R$. Since $P$ is $\HH$-stable, so is $J$. We also observe that $J$ is a Poisson ideal of $A$. Namely, if $a\in J$ and $f\in A$, then $aX+e \in P$ for some $e\in A$, whence $\{ aX+e, f \} \in P$. Since
\begin{align*}
\{ aX+e, f \} &= \{ a,f \} X + a \bigl( \sig(f) X + \de(f) \big) + \{ e,f \}  \\
&= \{ a,f \} X + \sig(f) (aX+e) + a \de(f) - \sig(f) e + \{ e,f \},
\end{align*}
we see that $\{ a,f \} X + (a \de(f) - \sig(f) e + \{ e,f \}) \in P$, so that $\{ a,f \} \in J$ as required. Thus, $J \cap Q$ is an $\HH$-stable Poisson ideal of $A$.

There exist prime ideals $P_1,\dots,P_r$ in $A$ minimal over $J \cap Q$ such that $P_1P_2 \cdots P_r$ is contained in $J \cap Q$. Since $J \cap Q$ is an $\HH$-stable Poisson ideal, each $P_i$ is a Poisson $\HH$-prime ideal of $A$ (Lemma \ref{Pprimeinfo}(a) and Observation \ref{obs.minHprime}(ii)). By hypothesis, each $P_i$ contains a Poisson-normal $\HH$-eigenvector $a_i$, and thus $a := a_1a_2 \cdots a_r$ is a Poisson-normal $\HH$-eigenvector of $A$ that lies in $J \cap Q$. Since $a$ is in $J$, it is the leading coefficient of an element of $P$ of degree $1$. By Lemma \ref{PnormalHeigen}(b), the coset $u := a+P$ is a Poisson-normal $\HH$-eigenvector of $R/P$. Moreover, $u \in Q/P$ because $a\in Q$.

Now assume that $P=0$. If $Q\cap A \ne 0$, then by hypothesis $Q \cap A$ contains a Poisson-normal $\HH$-eigenvector $a$ of $A$. Then $\de^l(a) \in Q \cap A$ for all $l \in \Znn$, whence the element $u := \theta(a) X^s$ lies in $Q$, where $s \in \Znn$ is minimal such that $\theta(a) X^s \in R$. By Lemma \ref{max.des(a)} and Proposition \ref{firstPnormal}, $u$ is a Poisson-normal $\HH$-eigenvector in $R$.

Finally, suppose that $Q \cap A = 0$. As above, the set
$$
J := \{ a\in A \mid aX+e \in Q\ \text{for some}\ e \in A \}
$$
is a nonzero $\HH$-stable Poisson ideal of $A$. If $J=A$, then $1\in J$, while if $J \ne A$, then $J$ contains a product of nonzero Poisson $\HH$-prime ideals of $A$. In either case, there is a Poisson-normal $\HH$-eigenvector $a$ of $A$ that lies in $J$. Let $s\in \Znn$ be maximal such that $\de^s(a) \ne 0$. 

If $s>0$, then by Proposition \ref{firstPnormal} and Lemma \ref{PnormalHeigen}(a), the element $u := \theta(a) X^s$ is a Poisson-normal $\HH$-eigenvector of $R$ that lies in $Q$. On the other hand, if $s=0$, Lemma \ref{PnormalHeigen}(b) shows that $\de \equiv 0$ and $Q = RX$. In this case, $u := X$ is a Poisson-normal $\HH$-eigenvector of $R$ that lies in $Q$.
\end{proof}

\subsection{Carrying Poisson-normal $\HH$-separation from $A$ to $R$}

Recall the concept of Poisson-normal $\HH$-separation from Definition \ref{PnormalHsep}.

\begin{theorem} \label{PnormalHsep.AtoR}
If $\HPspec A$ has Poisson-normal $\HH$-separation, then so does $\HPspec R$.
\end{theorem}

\begin{proof} Let $P \subsetneq Q$ be a proper inclusion of Poisson $\HH$-prime ideals of $R$. Then $P_0 := P \cap A$ is a $\de$-stable Poisson $\HH$-prime ideal of $A$ (Lemma \ref{PHcontract}(a)), and we may replace $A$, $R$, $P$, $Q$ by $A/P_0$, $R/RP_0$, $P/RP_0$, $Q/RP_0$, respectively. Thus, there is no loss of generality in assuming that $P \cap A = 0$.

The hypothesis of Poisson-normal $\HH$-separation now implies that every nonzero Poisson $\HH$-prime ideal of $A$ contains a Poisson-normal $\HH$-eigenvector of $A$. Therefore, by Proposition \ref{Pnormalsep1}, there exists a Poisson-normal $\HH$-eigenvector $u$ of $R/P$ such that $u \in Q/P$. This verifies Poisson-normal $\HH$-separation in $\HPspec R$.
\end{proof}

%%%%%%%%%%%%%%%%%%%%%%%%%
\sectionnew{Proof of the Main Theorem}

The Main Theorem \ref{mainthm} follows easily from Theorems \ref{PnormalHsep.AtoR} and \ref{HPnormsep>Pnormsep} together with Corollary \ref{affine+Pnormsep>Pcat}. In fact, these results yield Poisson catenarity for a somewhat larger class of algebras, as we now show. The following lemma will line up an inductive step.

\begin{lemma}  \label{lifttoPCauchon}
Let $R$ be a Poisson $\KK$-algebra, equipped with a rational Poisson action of an algebraic $\KK$-torus $\HH$, and let $A$ be an $\HH$-stable Poisson subalgebra of $R$. Assume that there exist an $\HH$-eigenvector $x \in R$ and an element $h_\circ \in \hfrak := \Lie \HH$ such that
\begin{enumerate}
\item[\rm(i)] The rule $\de(a) = \{ x,a \} - (h_\circ \cdot a) x$ defines a map $\de$ from $A$ to itself.
\item[\rm(ii)] $\de$ is locally nilpotent.
\item[\rm(iii)] The $h_\circ$-eigenvalue of $x$ is nonzero.
\end{enumerate}
Then there exists a Poisson-Cauchon extension $R^+ = A[X;(h_\circ \cdot)|_A,\de]_p$ relative to the same torus $\HH$, and the identity map on $A$ extends to an $\HH$-equivariant Poisson algebra homomorphism $\phi : R^+ \rightarrow R$ such that $\phi(X) = x$.
\end{lemma}

\begin{proof} Recall (Proposition \ref{h.action}(a)) that $\sig := (h_\circ \cdot)|_A$ is a Poisson derivation on $A$. Since also $\{x, - \}$ is a derivation, the map $\de$ is a derivation on $A$. As shown in the proof of \cite[Proposition 1.8]{KGSL}, $\de$ satisfies condition \eqref{delta.cond}, and thus there exists a Poisson polynomial ring $R^+ = A[X; \sig ,\de]_p$. By construction, the identity map on $A$ extends to a Poisson algebra homomorphism $\phi : R^+ \rightarrow R$ such that $\phi(X) = x$.

The action of $\HH$ on $A$ extends to a rational action of $\HH$ on $R^+$ -- by $\KK$-algebra automorphisms, at least -- such that $X$ is an $\HH$-eigenvector with the same $\HH$-eigenvalue as $x$. With respect to this action, $\phi$ is $\HH$-equivariant. For $\al \in \HH$ and $a\in A$, we compute that
\begin{align*}
\al\cdot \{ X,a \} &= \al \cdot \bigl(\sig(a) X + \de(a) \bigr) = \al \cdot \bigl( (h_\circ\cdot a) X + \{ x,a \} - (h_\circ \cdot a)x  \bigr)  \\
&= \bigl( h_\circ \cdot (\al \cdot a) \bigr) \chi_{x}(\al) X + \{ \chi_{x}(\al) x,\, \al \cdot a \} - \bigl( h_\circ \cdot (\al \cdot a) \bigr) \bigl( \chi_{x}(\al) x \bigr)  \\
&= \chi_{x}(\al) \bigl[ \sig(\al \cdot a) X + \de( \al \cdot a) \bigr] = \chi_{x}(\al) \{ X,\, \al \cdot a \} = \{ \al \cdot X ,\, \al \cdot a \},
\end{align*}
where we have used Proposition \ref{h.action}(b).
Since also $\al \cdot \{ a,b \} = \{ \al \cdot a ,\, \al \cdot b \}$ for all $a,b \in A$, it follows that $\al$ preserves the Poisson bracket on $R^+$.
Therefore the action of $\HH$ on $R^+$ is a Poisson action.

By assumption (iii), $h_\circ \cdot x = \la_\circ x$ for some nonzero $\la_\circ \in \KK$. Since $X$ and $x$ are $\HH$-eigenvectors with the same $\HH$-eigenvalue, they are also $\hfrak$-eigenvectors with the same $\hfrak$-eigenvalue, by Proposition \ref{h.action}(c)(d). In particular, $h_\circ \cdot X = \la_\circ X$. 

Therefore $R^+$ is a Poisson-Cauchon extension relative to $\HH$.
\end{proof}

\begin{theorem}  \label{genPnilp}
Let $R$ be a Poisson $\KK$-algebra, equipped with a rational Poisson action of an algebraic $\KK$-torus $\HH$. Assume that $R$ is generated {\rm(}as a $\KK$-algebra{\rm)} by $\HH$-eigenvectors
$x_1,\dots,x_N$, and that there exist $h_1,\dots,h_N \in \hfrak := \Lie \HH$
such that
\begin{enumerate}
\item[\rm(i)] $\{x_i,x_j\}- (h_i \cdot x_j)x_i \in \KK\langle x_1,\dots,x_{i-1} \rangle$
for $N \ge i >j \ge 1$.
\item[\rm(ii)] For all $i \in \gc2,N\dc$, the map $a \mapsto \{ x_i,a \} - (h_i \cdot a) x_i$ on $\KK\langle x_1,\dots,x_{i-1} \rangle$ is locally nilpotent.
\item[\rm(iii)] For all $i \in \gc1,N\dc$, the $h_i$-eigenvalue of $x_i$ is nonzero.
\end{enumerate}
Then $\Pspec R$ has Poisson-normal separation and is catenary.
\end{theorem}

\begin{remark*}  \label{dei}
For $i \in \gc2,N\dc$, the map $\de_i := \bigl( a \mapsto \{ x_i,a \} - (h_i \cdot a) x_i \bigr)$ in hypothesis (ii) is a derivation from $\KK\langle x_1,\dots,x_{i-1} \rangle$ to $R$, because $\{ x_i, - \}$ and $(h_i\cdot)$ are derivations. In view of hypothesis (i), it follows that $\de_i$ maps $\KK\langle x_1,\dots,x_{i-1} \rangle$ to itself. Thus, hypothesis (ii) says that $\de_i$ is a locally nilpotent derivation on $\KK\langle x_1,\dots,x_{i-1} \rangle$.
\end{remark*}

\begin{proof} We show that $\HPspec R$ has Poisson-normal $\HH$-separation. Once that is established, Theorem \ref{HPnormsep>Pnormsep} will imply that $\Pspec R$ has Poisson-normal separation, and then Corollary \ref{affine+Pnormsep>Pcat} will imply that $\Pspec R$ is catenary.

It suffices to show that Poisson-normal $\HH$-separation holds in $\HPspec(R/P)$ for any factor algebra $R/P$ where $P \in \HPspec R$. These factor algebras are also Poisson $\KK$-algebras, $\HH$ acts rationally on them by Poisson automorphisms, and they are generated by $\HH$-eigenvectors satisfying conditions corresponding to (i)--(iii). Thus, there is no loss of generality in assuming that $R$ is a domain.

We now prove Poisson-normal $\HH$-separation for $\HPspec R$ by induction on $N$, starting with the case $N=0$. In that case, $R=\KK$ and the condition holds trivially.
Now let $N>0$ and assume that $\HPspec$ of the algebra $A := \KK \langle x_1,\dots,x_{N-1} \rangle$ has Poisson-normal $\HH$-separation. Note that $A$ is stable under the $\HH$-action and hence also under the $\hfrak$-action.

In view of Lemma \ref{lifttoPCauchon}, there exist a Poisson-Cauchon extension 
$R^+ := A[X; \sig, \de]_p$ relative to $\HH$ and a surjective $\HH$-equivariant Poisson algebra homomorphism $\phi : R^+ \rightarrow R$.  

Since $A$ is a noetherian domain and $\HPspec A$ has Poisson-normal $\HH$-separation, Theorem \ref{PnormalHsep.AtoR} implies that $\HPspec R^+$  has Poisson-normal $\HH$-separation. Therefore $\HPspec R$ has Poisson-normal $\HH$-separation, completing the induction step.
\end{proof}

%%%%%%%%%%%%%%%%%%%%%% References %%%%%%%%%%%%%%%%%%%%%%%%%%%%%%%%%%%%%%%

%%%%%%%%%%%%%%%%%%%%%%%%%%%%%%%%%%%%%%%%%%%%%%%%%%%%%%%%%%%%%%%%%%%%%%%%%%%%%%%

\end{document}